\documentclass{article}
\usepackage[utf8]{inputenc}
\usepackage[hidelinks]{hyperref}
\usepackage[marginratio={1:1, 1:1}]{geometry}

\usepackage[T1]{fontenc}
\usepackage{etex,microtype}

\usepackage{amsmath,amssymb}
\usepackage{amsthm,mathrsfs}
\usepackage[all, cmtip]{xy}
\usepackage{mathtools}
\usepackage{enumitem}
\usepackage{centernot}

\title{Marked and labelled Gushel--Mukai fourfolds}

\author{Emma Brakkee and Laura Pertusi}

\date{\today}

\DeclareMathOperator{\tM}{M}
\DeclareMathOperator{\tO}{O}
\newcommand{\stO}{\widetilde{\tO}}

\DeclareMathOperator{\modulo}{mod}
\renewcommand{\mod}{\:\modulo \:}
\DeclareMathOperator{\rk}{rk}
\newcommand{\mD}{\mathcal{D}}
\newcommand{\mM}{\mathcal{M}}
\newcommand{\mO}{\mathcal{O}}

\DeclareMathOperator{\Disc}{Disc}
\DeclareMathOperator{\id}{id}
\DeclareMathOperator{\HH}{H}

\DeclareMathOperator{\MathOpHom}{Hom}
\renewcommand{\hom}{\MathOpHom}
\DeclareMathOperator{\MathOpKer}{Ker}
\renewcommand{\ker}{\MathOpKer}
\DeclareMathOperator{\im}{Im}

\newcommand{\Z}{\mathbb{Z}}
\newcommand{\Q}{\mathbb{Q}}
\newcommand{\C}{\mathbb{C}}
\newcommand{\p}{\mathbb{P}}
\newcommand{\G}{\mathbb{G}}

\newcommand{\Gr}{\text{Gr}}
\newcommand{\Ku}{\text{Ku}}
\newcommand{\D}{\text{D}^b}

\DeclareMathOperator{\mar}{mar}
\DeclareMathOperator{\lab}{lab}

\DeclareMathOperator{\br}{Br}

\DeclareMathOperator{\tors}{tors}
\DeclareMathOperator{\ns}{NS}
\DeclareMathOperator{\pic}{Pic}
\DeclareMathOperator{\prim}{pr}
\DeclareMathOperator{\stab}{Stab}
\newcommand{\sheafhom}{\mathscr{H}\kern -.5pt om}

\def\dual{\smash{\raisebox{-0.1em}{\scalebox{.7}[1.4]{\rotatebox{90}{\textnormal{\guilsinglleft}}}}}}

\theoremstyle{plain}
\newtheorem{theorem}{Theorem}[section]
\newtheorem{proposition}[theorem]{Proposition}
\newtheorem{lemma}[theorem]{Lemma}
\newtheorem{cor}[theorem]{Corollary}

\theoremstyle{definition}
\newtheorem{definition}[theorem]{Definition}
\newtheorem{remark}[theorem]{Remark}
\newtheorem*{plan}{Plan of the paper}
\newtheorem*{notation}{Notation}
\newtheorem*{ack}{Acknowledgements}

\begin{document}

\maketitle
\begin{abstract}
We prove that the moduli stacks of marked and labelled Hodge-special Gushel--Mukai fourfolds are isomorphic.
As an application, we construct rational maps from the stack of Hodge-special Gushel--Mukai fourfolds of discriminant $d$ to the moduli space of (twisted) degree-$d$ polarized K3 surfaces. We use these results to prove a counting formula for the number of $4$-dimensional fibers of Fourier--Mukai partners of very general Hodge-special Gushel--Mukai fourfolds with associated K3 surface, and a lower bound for this number in the case of a twisted associated K3 surface.
\end{abstract}

\section{Introduction}
In the last 20 years the study of cubic fourfolds has been a central research topic, due for instance to their rich associated hyperk\"ahler geometry and the still open question concerning whether they are rational or irrational. One foundational work is \cite{Hassett}, where Hassett studied special cubic fourfolds, i.e.\ cubic fourfolds containing a surface which is not a complete intersection. Special cubic fourfolds form divisors in the moduli space of cubic fourfolds parametrized by a positive even integer $d$ called the discriminant. Moreover, depending on the value of 
$d$, the cubic fourfold is related to a degree-$d$ polarized K3 surface via Hodge theory.

In order to study this relation on the level of period domains and moduli spaces, Hassett introduced the notions of marked and labelled special cubic fourfolds. Depending on $d$, the moduli space of discriminant $d$ marked cubic fourfolds is either isomorphic to or a two-to-one covering of the moduli space of discriminant $d$ labelled cubic fourfolds. Moreover, if $d$ is such that an associated K3 surface exists, this is used to construct an either generically injective or degree-two rational map from the moduli space of degree-$d$ polarized K3 surfaces to the divisor of discriminant-$d$ special cubic fourfolds. This difference was further investigated in \cite{BrakkeeTwoK3}, where the geometry of the covering involution arising in the second case is completely described.

\medskip
In this paper, we deal with similar questions in the case of Gushel--Mukai fourfolds. These are smooth Fano fourfolds obtained generically as quadric sections of linear sections of the Grassmannian $\Gr(2,5)$, and they share many similarities with cubic fourfolds.  After defining marked and labelled GM fourfolds of discriminant $d$ and their associated moduli stacks in Section \ref{section-markedlabelled}, we show that they provide equivalent notions in this case. 

\begin{theorem}[Corollary \ref{IsoModStacks}]
\label{thm_isomodstacks}
The moduli stacks of labelled and marked Hodge-special GM fourfolds are isomorphic.
\end{theorem}

Like for cubic fourfolds, this result is particularly interesting when we specify to GM fourfolds with Hodge-associated K3 surfaces, as defined in \cite{debarre_iliev_manivel_2015}. Recall that a GM fourfold has an associated K3 surface if and only if its discriminant satisfies a certain numerical condition $\eqref{eq_astast}$ -- see Sections \ref{section_HspGM}
and \ref{DefRationalMap}. For these values of the discriminant, applying Theorem \ref{thm_isomodstacks}, we interpret the condition of having an associated K3 surface on the level of moduli stacks as follows.

\begin{theorem}
\label{cor_ratmap}
Let $d$ be a positive integer satisfying condition $\eqref{eq_astast}$. Then there exists a dominant rational map defined in \eqref{eq_ratmap} from the moduli stack of Hodge-special GM fourfolds with discriminant $d$ to the moduli space of degree-$d$ polarized K3 surfaces that sends a GM fourfold to a Hodge-associated K3 surface.
\end{theorem}

As an application, we can count fibers of the period map for GM fourfolds whose elements are Fourier--Mukai partners. By \cite{kuznetsov_perry} the bounded derived category of a GM fourfold $X$ has a semiorthogonal decomposition of the form
\begin{equation*}
\D(X)= \langle \Ku(X), \mathcal{O}_X, \mathcal{U}_X^*, \mathcal{O}_X(1), \mathcal{U}_X^*(1)\rangle,    
\end{equation*}
where $\mathcal{U}_X^*$ is the restriction to $X$ of the tautological rank-$2$ vector bundle on $\Gr(2,5)$ and $\Ku(X)$, defined as the orthogonal complement to the exceptional collection $\mathcal{O}_X, \mathcal{U}_X^*, \mathcal{O}_X(1), \mathcal{U}_X^*(1)$, is a subcategory of K3 type. We say that a GM fourfold $X'$ is a Fourier--Mukai partner of $X$ if there is an equivalence $\Ku(X) \xrightarrow{\sim} \Ku(X')$ of Fourier--Mukai type. As shown in \cite[Theorem 4.4]{debarre_iliev_manivel_2015} the period map of GM fourfolds has smooth $4$-dimensional fibers, so we cannot expect a finite number of Fourier--Mukai partners as in the case of K3 surfaces \cite{BriMac} or cubic fourfolds \cite[Theorem 1.1]{Huy}. Nevertheless, Theorem \ref{cor_ratmap} allows to prove a counting formula to the number of period points of Fourier--Mukai partners for very general GM fourfolds with Hodge-associated K3 surface. See \cite{Oguiso} and \cite{Pert1} for the analogous statements for K3 surfaces and cubic fourfolds, respectively.

\begin{theorem}[Proposition \ref{prop_FMpGM}]
\label{thm_FMp}
Let $X$ be a very general Hodge-special GM fourfold with discriminant $d$ satisfying \eqref{eq_astast}. Let $m$ be the number of non-isomorphic Fourier--Mukai partners of its Hodge-associated K3 surface. Then when $d \equiv 4 \mod 8$ (resp.\ $d \equiv 2 \mod 8$), there are $m$ (resp.\ $2m$) fibers of the period map of GM fourfolds such that, when non-empty, their elements are Fourier--Mukai partners of $X$. Moreover, all Fourier--Mukai partners of $X$ are obtained in this way.
\end{theorem}

We end with proving the analogue of Theorem \ref{cor_ratmap} for GM fourfolds with associated twisted K3 surface. Recall that by \cite[Theorem 1.1]{Pert2} this is equivalent to having discriminant of the form $d'=dr^2$ with $d$ satisfying $(\ast\ast)$ -- see Section \ref{section_GMvstwistedK3}. On the other hand, the moduli space of polarized twisted K3 surfaces with fixed degree and order was recently constructed in \cite{BrakkeeTwistedK3}. 

\begin{theorem}[Corollary \ref{RatMapTwisted}]
\label{thm_ratmaptwisted}
Let $d'$ be a positive integer such that a very general GM fourfold of discriminant $d'$ admits an associated polarized twisted K3 surface of degree $d$ and order $r$. There is a dominant rational map from the moduli stack of Hodge-special GM fourfolds of discriminant $d'$ to a component of the moduli space of twisted K3 surfaces of degree $d$ and order $r$, sending a GM fourfold of discriminant $d'$ to an associated twisted K3 surface.
\end{theorem}

Finally, as in the untwisted setting, we apply Theorem \ref{thm_ratmaptwisted} to study Fourier--Mukai partners of a very general GM fourfold with associated twisted K3 surface.
\begin{theorem}[Proposition \ref{prop_FMptwistedcase}]
\label{thm_FMptwisted}
Let $d'$ be a positive integer such that a very general GM fourfold of discriminant $d'$ admits an associated polarized twisted K3 surface $(S,\alpha)$ of degree $d$ and order $r$. Let $m'$ be the number of non-isomorphic Fourier--Mukai partners of $(S,\alpha)$ of order $r$.
Then when $d' \equiv 0 \mod 4$ (resp.\ $d' \equiv 2 \mod 8$), there are at least $m'$ (resp.\ $2m'$) fibers of the period map of GM fourfolds such that, when non-empty, their elements are Fourier--Mukai partners of $X$.
\end{theorem}

\begin{plan}
In Section \ref{section_introGM} we recall the definition of (Hodge-special) GM fourfolds and some results concerning their Hodge theory. In Section \ref{section-markedlabelled} we define marked and labelled Hodge-special GM fourfolds and we prove Theorem \ref{thm_isomodstacks}. Section \ref{section_GMvsK3} is devoted to the construction of the rational map of Theorem \ref{cor_ratmap} and the proof of Theorem \ref{thm_FMp}. Finally, in Section \ref{section_GMandtwistedK3} we recall the construction of moduli spaces of twisted K3 surfaces with fixed order and degree, and we prove Theorems \ref{thm_ratmaptwisted} and \ref{thm_FMptwisted}.
\end{plan}

\begin{notation}
Given a lattice $L$, we denote by $\Disc L:=L^{\dual}/L$ its discriminant group and we set $\stO(L) := \ker(\tO(L) \to \tO(\Disc L))$. For any integer $m \neq 0$ we denote by $L(m)$ the lattice $L$ with the intersection form multiplied by $m$. 

We denote by $I_1$ the lattice $\Z$ with bilinear form $(1)$, $I_{r,s}:=I_1^{\oplus r} \oplus I_1(-1)^{\oplus s}$, $A_1:=I_1(2)$, $U$ is the hyperbolic plane $\left(\Z^{\oplus 2},\bigl(\begin{smallmatrix}
0 & 1 \\ 
1 & 0
\end{smallmatrix}\bigr)\right)$ and $E_8$ is the unique even unimodular lattice of signature $(8,0)$. 

For $3 \geq i \geq j \geq 0$ the Schubert cycles on the Grassmannian $\Gr(2,5)$ are denoted by $\sigma_{i,j} \in \HH^{2(i+j)}(\Gr(2,5),\Z)$ and we set $\sigma_i:=\sigma_{i,0}$. 
\end{notation}

\begin{ack}
We thank Gerard van der Geer and Mingmin Shen for their interest, and Thorsten Beckmann for useful discussions. We are grateful to Daniel Huybrechts, Alex Perry and Paolo Stellari for suggestions on the preliminary version of this work.

This work started when the second author was visiting the Max-Planck-Institut f\"ur Mathematik in Bonn whose hospitality is gratefully acknowledged.

The first author is supported by
NWO Innovational Research Incentives Scheme 016.Vidi.189.015. The second author is supported by the ERC Consolidator Grant ERC-2017-CoG-771507, Stab-CondEn.
\end{ack}

\section{Gushel--Mukai fourfolds}
\label{section_introGM}
In this section, we review the definition of Gushel--Mukai fourfolds and some known results concerning their Hodge theory. Our main references are \cite{debarre_iliev_manivel_2015, debarre_kuznetsov_2019}. We assume the base field is $\C$.

\subsection{Cohomology and period domain of Gushel--Mukai fourfolds}
Let $V_5$ be a $5$-dimensional $\C$-vector space and denote by $\text{CGr}(2,V_5)$ the cone over the Grassmannian $\Gr(2,V_5)$ with vertex $\nu:=\p(\C)$, embedded in $\p(\C \oplus \bigwedge^2 V_5) \cong \p^{10}$ via the Pl\"ucker embedding of $\Gr(2,V_5) \subset \p(\bigwedge^2 V_5)$.

\begin{definition}
A \emph{Gushel--Mukai (GM) fourfold} is a smooth $4$-dimensional intersection
$$X:=\text{CGr}(2,V_5) \cap Q$$
where $Q \subset \p(W)$ is a quadric hypersurface in a linear space $\p(W) \cong \p^8 \subset \p(\C \oplus \bigwedge^2 V_5)$. 
\end{definition}

Since $X$ is smooth, the linear projection $\gamma_X\colon X \to \Gr(2,V_5)$ from the vertex $\nu$ is a regular map. 
The restriction of the hyperplane class on $\p(\C \oplus \bigwedge^2 V_5)$ defines a natural polarization $H:=\gamma_X^*\sigma_1$ on $X$ with degree $H^4=10$. By the adjunction formula, the canonical divisor is $K_X=-2H$, so $X$ is a Fano fourfold of degree $10$ and index $2$.

The moduli stack $\mM_4$ of GM fourfolds is a smooth, irreducible Deligne--Mumford stack of finite type over $\C$ of dimension $24$ \cite[Prop.\ 2.4]{kuznetsov_perry}.

\medskip
By \cite[Lemma 4.1]{Iliev_Maniv} the Hodge diamond of $X$ is
\[
\begin{tabular}{ccccccccccccccc}
&&&&&&&1\\
&&&&&&0&&0&\\
&&&&&0&&1&&0\\
&&&&0&&0&&0&&0\\
&&&0&&1&&22&&1&&0.
\end{tabular}
\]
By \cite[Proposition 5.1]{debarre_iliev_manivel_2015} there is an isomorphism of lattices
$$H^4(X,\Z) \cong \Lambda:= I_{22,2}.$$
Note that the rank-2 lattice $\HH^4(\Gr(2,V_5),\Z)$ embeds into $\HH^4(X,\Z)$ via $\gamma_X^*$. The \emph{vanishing lattice} of $X$ is the sublattice
$$\HH^4(X,\Z)_{00}:=\left\lbrace x \in \HH^4(X,\Z) \mid x \cdot \gamma_X^{*}(\HH^4(\Gr(2,V_5),\Z))=0 \right\rbrace.$$
By \cite[Proposition 5.1]{debarre_iliev_manivel_2015} it is isomorphic to 
$$\Lambda_{00}:=E_8^{\oplus 2} \oplus U^{\oplus 2} \oplus A_1^{\oplus 2}.$$
Note that the intersection form on $\gamma_X^*(\HH^4(\Gr(2,V_5),\Z))$ with respect to the basis $\gamma_X^*\sigma_{1,1}, \gamma_X^*\sigma_2$ is represented by the matrix $\begin{pmatrix}
2 & 2\\
2 & 4
\end{pmatrix}$. Fixing a primitive embedding of $\Lambda_{00}$ into $\Lambda$, we set $\Lambda_G:=\Lambda_{00}^{\perp}\subset\Lambda$ and we can find two generators $\lambda_1$ and $\lambda_2$ of $\Lambda_G$ such that the intersection matrix is
$\begin{pmatrix}
2 & 0\\
0 & 2
\end{pmatrix}$.

The period domain of GM fourfolds is the complex manifold 
\begin{equation*}
\label{locperiodom}
\Omega(\Lambda_{00}):= \lbrace w \in \p(\Lambda_{00} \otimes \C) \mid w \cdot w =0, w \cdot \bar{w}<0 \rbrace.
\end{equation*}
Note that the group $\stO(\Lambda_{00})$ acts properly discontinuously on $\Omega(\Lambda_{00})$ and it is isomorphic to
\[\Gamma:=\{g\in\tO(\Lambda)\mid g|_{\Lambda_G}=\id_{\Lambda_G}\}.\]
The quotient 
$$\mD:=\Omega(\Lambda_{00})/\stO(\Lambda_{00})$$
is an irreducible quasi-projective variety of dimension $20$ and by \cite[Theorem 4.4]{debarre_iliev_manivel_2015} the period map $p\colon \mM_4 \to \mD$ is dominant as a map of stacks with smooth $4$-dimensional fibers. The period point of $X$ is $p(X) \in \mD$.

\subsection{Hodge-special Gushel--Mukai fourfolds}
\label{section_HspGM}
A very general GM fourfold $X$ satisfies $\rk \HH^{2,2}(X,\Z)=2$.
We call $X$ \emph{Hodge-special} if $H^{2,2}(X,\Z)$ contains a rank-three primitive sublattice containing $\gamma_X^*(H^4(\Gr(2,V_5),\Z))$. 

Period points of Hodge-special GM fourfolds lie in codimension-$1$ Noether--Lefschetz loci in $\mD$. 
Indeed, let $L_d \subset \Lambda$ be a primitive rank-three positive definite sublattice containing $\Lambda_G$, with discriminant $d$. By \cite[Lemma 6.1]{debarre_iliev_manivel_2015} we have $d \equiv 0,2$ or $4 \mod 8$.
Consider the codimension-$1$ locus
$$\Omega(L_d^\perp):=\p(L_d^\perp \otimes \C) \cap \Omega(\Lambda_{00})$$
where $L_d^\perp$ is the orthogonal complement of $L_d$ in $\Lambda$. Let 
$$\mD_{L_d}\subset\mD$$
be the image of $\Omega(L_d^{\perp})$ under the 
map $\Omega(\Lambda_{00})\to \mD$. Then the period of any Hodge-special GM fourfold lies in $\mD_{L_d}$ for some $L_d$.

By \cite[Proposition 6.2]{debarre_iliev_manivel_2015}, 
the lattice $L_d$ only depends on the discriminant $d$, and depending on $d$, there are one or two embeddings of $L_d$ into $\Lambda$ up to composition with elements of $\stO(\Lambda_{00})$.
To be precise, up to the action of $\stO(\Lambda_{00})$, there exists $\tau \in L_d$ such that $\lambda_1, \lambda_2, \tau$ is a basis for $L_d$ with intersection matrix given by
\[
\begin{pmatrix}
2 & 0 & 0 \\ 
0 & 2 & 0 \\ 
0 & 0 & 2k
\end{pmatrix} \quad \text{if }d=8k,
\]
\[
\begin{pmatrix}
2 & 0 & 1 \\ 
0 & 2 & 0 \\ 
1 & 0 & 2k
\end{pmatrix} \quad \text{or} \quad 
\begin{pmatrix}
2 & 0 & 0 \\ 
0 & 2 & 1 \\ 
0 & 1 & 2k
\end{pmatrix} \quad \text{if }d=2+8k,
\]
\[
\begin{pmatrix}
2 & 0 & 1 \\ 
0 & 2 & 1 \\ 
1 & 1 & 2k
\end{pmatrix} \quad \text{if }d=4+8k.
\]
In the case $d=2+8k$, denote by $\mD'_d$ and $\mD''_d$ the divisors $\mD_{L_d}$ corresponding to the first and second embedding of $L_d$, respectively. 
It follows from \cite[Corollary 6.3]{debarre_iliev_manivel_2015} that the periods of Hodge-special GM fourfolds are contained in the union of 
\begin{enumerate}[noitemsep,label=(\roman*)]
    \item the irreducible hypersurfaces $\mD_d:=\mD_{L_d}\subset \mD$ for all $d\equiv 0\mod 4$;
    \item the unions $\mD_d:=\mD'_d\cup\mD''_d$ for all $d\equiv 2\mod 8$.
\end{enumerate}
Moreover, there exists an involution $r \in \tO(\Lambda_{00})$ which is not in $\stO(\Lambda_{00})$, inducing an involution $r_{\mD}$ on $\mD$ which exchanges $\mD_d'$ and $\mD_d''$ when $d \equiv 2 \mod 8$.

\medskip
We say that a Hodge-special GM fourfold $X$ has discriminant $d$ if its period point belongs to $\mD_d$.
The moduli stack of Hodge-special GM fourfolds of discriminant $d$ is $\mM_4\times_{\mD}\mD_d$. Note that a very general $X \in \mM_4\times_{\mD}\mD_d$ satisfies $\text{rk}\HH^{2,2}(X,\Z)=3$. 

It is known that each of the irreducible divisors $\mD_{L_d}$ intersects the image of $p$ for $d > 8$ \cite[Theorem 8.1]{debarre_iliev_manivel_2015}, so   $\mD_{L_d}\cap\im(p)$ contains an open dense subset of $\mD_{L_d}$.
It follows that the restriction $p\colon \mM_4\times_{\mD}\mD_d\to\mD_d$ is still dominant when $d>8$.

\section{Marked and labelled Gushel--Mukai fourfolds}
\label{section-markedlabelled}
In analogy to \cite[Definition 3.1.3]{Hassett} for cubic fourfolds, we give the following definition. Let $L_d$ be a rank-$3$ positive definite lattice containing $\Lambda_G$. 
\begin{definition}
A \emph{marked} Hodge-special GM fourfold is a GM fourfold $X$ together with a primitive embedding $\varphi\colon L_d\hookrightarrow \HH^{2,2}(X,\Z)$ preserving the classes $\lambda_1$ and $\lambda_2$. 
A \emph{labelled} Hodge-special GM fourfold is a GM fourfold $X$ together with a primitive sublattice $L_d\subset \HH^{2,2}(X,\Z)$.
\end{definition}
So a labelling of a GM fourfold is the image of a marking.

\medskip
Two marked GM fourfolds 
$(X,\varphi\colon L_d\hookrightarrow \HH^{2,2}(X,\Z))$ and ${(X',\varphi'\colon L_d\hookrightarrow \HH^{2,2}(X',\Z))}$ 
are isomorphic if there is an isomorphism $f\colon X\to X'$ such that $f^*\colon \HH^4(X',\Z)\to \HH^4(X,\Z)$ satisfies
$f^*\circ \varphi = \varphi'$.
Two labelled GM fourfolds 
$(X,L_d\subset \HH^{2,2}(X,\Z))$ and ${(X', L_d\subset \HH^{2,2}(X',\Z))}$ 
are isomorphic if there exists an isomorphism $f\colon X\to X'$ such that $f^*$ preserves $L_d$.

\begin{remark}
Consider the sets of isomorphism classes of marked and labelled GM fourfolds. There is a map
\[\{\text{ marked GM 4-folds }\}/_{\cong} \to \{\text{ labelled GM 4-folds }\}/_{\cong}\]
sending $(X,\varphi\colon L_d\hookrightarrow \HH^{2,2}(X,\Z))$ to 
$(X,\varphi(L_d)\subset \HH^{2,2}(X,\Z))$. It is surjective but need, a priori, not be injective: 
The lattice $L_d$ could have non-trivial automorphisms fixing the $\lambda_i$.
\end{remark}

Fix an embedding $L_d\hookrightarrow \Lambda$.
Recall that $\mD_{L_d}$ is the image of $\Omega(L_d^{\perp})$ under 
${\Omega(\Lambda_{00})\to\mD=\Omega(\Lambda_{00})/\Gamma}$.
Let
\begin{align*}
G(L_d) &:= \{g\in\Gamma : g(L_d) = L_d\}\\
H(L_d) &:= \{g\in G(L_d): g|_{L_d} = \id_{L_d}\}
\end{align*}
and define
\begin{align*}
\mD_{L_d}^{\lab} &:= \Omega(L_d^{\perp})/G(L_d)\\
\mD_{L_d}^{\mar} &:= \Omega(L_d^{\perp})/H(L_d).
\end{align*}
Then we have surjective maps
\[\mD_{L_d}^{\mar}\to \mD_{L_d}^{\lab}\to\mD_{L_d}\subset\mD.\]
When $d\equiv 0\mod 4$, we set $\mD_d^{\lab}:=\mD_{L_d}^{\lab}$ and 
$\mD_d^{\mar}:=\mD_{L_d}^{\mar}$. When $d\equiv 2\mod 8$, we have two embeddings $\mD_{L_d}\xrightarrow{\cong}\mD_d'\subset\mD$ and $\mD_{L_d}\xrightarrow{\cong}\mD_d''\subset\mD$; let $(\mD_d')^{\lab}$ and $(\mD_d'')^{\lab}$ be the corresponding spaces $\mD_{L_d}^{\lab}$ over $\mD_d'$ and $\mD_d''$, respectively.
Note that if $x\in\mD_d'\cap\mD_d''$, then there are two embeddings of $L_d$ into the (2,2)-part of the corresponding Hodge structure on $\Lambda_{00}$ that are in different $\stO(\Lambda_{00})$-orbits. So $x$ has two labellings, giving rise to one point in $(\mD_d')^{\lab}$ and one in $(\mD_d'')^{\lab}$. Accordingly, we let $\mD_d^{\lab}$ be the disjoint union
\[\mD_d^{\lab}:=(\mD_d')^{\lab}\coprod(\mD_d'')^{\lab}.\]
Analogously, define $\mD_d^{\mar}:=(\mD_d')^{\mar}\coprod(\mD_d'')^{\mar}$.
 Then the moduli stacks of labelled and marked Hodge-special GM fourfolds of discriminant $d$ are $\mM_4\times_{\mD}\mD_d^{\lab}$ and 
$\mM_4\times_{\mD}\mD_d^{\mar}$, respectively.

\medskip
In the rest of this section, we analyze the natural surjective morphisms
\[\mD_{L_d}^{\mar}\to \mD_{L_d}^{\lab}\to\mD_{L_d}.\]
\begin{lemma}
\label{lemma_mapnu}
The natural map $\nu\colon \mD_{L_d}^{\lab}\twoheadrightarrow\mD_{L_d}$ is a normalization.
\end{lemma}
\begin{proof}
The argument is the same as in the case of cubic fourfolds in \cite[Section~2.3]{BrakkeeTwoK3}.
\end{proof}

Note that a non-normal point in $\mD_{L_d}$ has two different labellings by $L_d$. In particular, the integral (2,2)-part of the corresponding Hodge structure has rank bigger than 3.

\begin{proposition}\label{MainResult}
The map $\mD_{L_d}^{\mar}\twoheadrightarrow\mD_{L_d}^{\lab}$ is an isomorphism.
\end{proposition}
It follows that $\mD_d^{\mar}\to\mD_d^{\lab}$ is an isomorphism.

\medskip
In order to prove Proposition \ref{MainResult}, we will show that $G(L_d)/H(L_d)\cong\Z/2\Z$ and it is generated by an element
that restricts to $-\id$ on $L_d^{\perp}$.
As $-\id_{L_d^{\perp}}$ induces the trivial action on $\Omega(L_d^{\perp})$,
we deduce the proof of Proposition \ref{MainResult}.

\medskip
Let $G'(L_d) := \{g\in\tO(L_d): g(\lambda_i) = \lambda_i\}$.
Then $G(L_d)/H(L_d)$ is isomorphic to $G'(L_d)$ via restriction to $L_d$.

\begin{lemma}
\label{lemma_descrofG'}
The group $G'(L_d)$ is $\Z/2\Z$, generated by an element that acts on $\Disc(L_d)$ as $-\id$.
\end{lemma}
\begin{proof}
Let $g\in G'(L_d)$.  
Assume $d=8k$, so $L_d$ has a basis $\lambda_1,\lambda_2,\tau$ with corresponding intersection matrix
\[\begin{pmatrix} 2 & 0 & 0 \\
         0 & 2 & 0 \\
         0 & 0 & d/4
        \end{pmatrix}\]
(see Section \ref{section_HspGM}). Then either $g(\tau)=\tau$, so $g=\id_{L_d}$, or
$g(\tau)=-\tau$. In the second case, $g$ acts on the discriminant group
\[\Z/2\Z\oplus\Z/2\Z\oplus\Z/(d/4)\Z = \left\langle \frac{\lambda_1}{2},\frac{\lambda_2}2,\frac{\tau}{d/4}\right\rangle\]
of $L_d$ by 
\[\left(\frac{\lambda_1}{2},\frac{\lambda_2}2,\frac{\tau}{d/4}\right)\mapsto
 \left(\frac{\lambda_1}{2},\frac{\lambda_2}2,-\frac{\tau}{d/4}\right)\equiv
-\left(\frac{\lambda_1}{2},\frac{\lambda_2}2,\frac{\tau}{d/4}\right).
\]

Next, assume $d=2+8k$, so $L_d$ is isomorphic to the lattice with basis $\lambda_1,\lambda_2,\tau$ and intersection matrix
\[\begin{pmatrix} 2 & 0 & 0 \\
         0 & 2 & 1 \\
         0 & 1 & (d+2)/4
        \end{pmatrix}\]
Write $g(\tau) = a\lambda_1+b\lambda_2+c\tau$.
It follows from $g(\lambda_i) = \lambda_i$ that $a=0$ and $c=1-2b$, and solving $(g(\tau))^2 = (\tau)^2$ gives
$(b-b^2)d=0$. Hence we either have $b=0$,
so $g=\id_{L_d}$, or $b=1$ and $c=-1$, so $g(\tau) = \lambda_2-\tau$.
In the second case, the action on the discriminant group
\[\Z/2\Z\oplus\Z/(d/2)\Z = \left\langle \frac{\lambda_1}{2},\frac{\lambda_2-2\tau}{d/2}\right\rangle\]
of $L_d$ is given by 
\[\left(\frac{\lambda_1}{2},\frac{\lambda_2-2\tau}{d/2}\right)\mapsto
 \left(\frac{\lambda_1}{2},\frac{-\lambda_2+2\tau}{d/2}\right)\equiv
 -\left(\frac{\lambda_1}{2},\frac{\lambda_2-2\tau}{d/2}\right).
\]

Finally, assume $d=4+8k$, there is a basis $\lambda_1,\lambda_2,\tau$ for $L_d$ with intersection matrix
\[\begin{pmatrix} 2 & 0 & 1 \\
         0 & 2 & 1 \\
         1 & 1 & (d+4)/4
        \end{pmatrix}
\]
and write $g(\tau) = a\lambda_1+b\lambda_2+c\tau$.
Now $g(\lambda_i) = \lambda_i$, implies $a=b$ and $c = 1-2a$, and solving $(g(\tau))^2 = (\tau)^2$ gives 
$(a-a^2)d=0$. Hence we either get $a=0$,
so $g=\id_{L_d}$, or $a=1$ and $c=-1$, so $g(\tau) = \lambda_1+\lambda_2-\tau$.
In the second case, the action on the discriminant group
\[\Z/d\Z = \left\langle \frac{\lambda_1+\lambda_2-2\tau}{d}\right\rangle\]
of $L_d$ is given by 
\[\frac{\lambda_1+\lambda_2-2\tau}{d}\mapsto
 \frac{\lambda_1+\lambda_2-2(\lambda_1+\lambda_2-\tau)}{d}=
 -\frac{\lambda_1+\lambda_2-2\tau}{d}. \qedhere
\]
\end{proof}

\begin{proof}[Proof of Proposition~\ref{MainResult}]
By Lemma \ref{lemma_descrofG'}, the generator $\gamma'$ of $G'(L_d)$ acts as $-\id$ on $\Disc L_d$. Then $-\id_{L_d^{\perp}}\oplus\gamma'$ extends to an element $\gamma$ of $\tO(\Lambda)$ by \cite[Corollary 1.5.2 and Proposition 1.6.1]{Nikulin}, which generates $G(L_d)/H(L_d)$. Since by definition $\gamma$ restricts to
$-\id$ on $L_d^{\perp}$, we conclude that $\gamma$ acts trivially on $\Omega(L_d^\perp)$. This implies the statement.
\end{proof}

As a direct consequence of Proposition \ref{MainResult}, we get the following identification between moduli stacks of marked and labelled Hodge-special GM fourfolds.

\begin{cor}\label{IsoModStacks}
We have an isomorphism $\mM_4\times_{\mD}\mD_d^{\mar} \cong \mM_4\times_{\mD}\mD_d^{\lab}$.
\end{cor}

\section{Gushel--Mukai fourfolds with associated K3 surface}
\label{section_GMvsK3}
In this section we prove Theorem \ref{cor_ratmap} and Theorem \ref{thm_FMp}.

\subsection{Rational maps to moduli spaces of K3 surfaces}\label{DefRationalMap}

The aim of this section is to construct the rational map of Theorem \ref{cor_ratmap}. Let $X$ be a Hodge-special GM fourfold whose period lies in $\mD_d$, that is, there are a rank-$3$ positive definite lattice $L_d$ of discriminant $d$ containing $\Lambda_G$ and a primitive embedding $L_d\hookrightarrow \HH^{2,2}(X,\Z)$. As in \cite[Section 6.2]{debarre_iliev_manivel_2015}, we say that a quasi-polarized K3 surface $(S,l)$ is \emph{Hodge-associated} to $X$ if there is a Hodge isometry
\[\HH^2(S,\Z)\supset l^{\perp}\cong L_d^{\perp}\subset\HH^4(X,\Z)\]
up to a sign and a Tate twist. In particular, $(S,l)$ has degree $d$. By \cite[Prop.~6.5]{debarre_iliev_manivel_2015} 
$X$ has a Hodge--associated quasi-polarized K3 surface if and only if $d$ satisfies 
\begin{equation}\tag{$\ast\ast$}
\label{eq_astast}
d\equiv 2,4\mod 8 \text{ and } p\centernot| d\text{ for every prime }p\equiv 3\mod 4.    
\end{equation}
Moreover, when the period does not lie in $\mD_d\cap\mD_8$, then the quasi-polarized K3 surface is actually polarized.

\medskip
Denote by $\Lambda_d:=E_8(-1)^{\oplus 2} \oplus U^{\oplus 2} \oplus I_1(-d)$ the lattice isomorphic to the primitive middle cohomology $\HH^2(S,\Z)_{\prim}:=l^{\perp}\subset\HH^2(S,\Z)$ of a polarized K3 surface $(S,l)$ of degree $d$. Then condition \eqref{eq_astast} on $d$ is equivalent to the existence of an isomorphism of lattices
$L_d^{\perp}\cong\Lambda_d(-1)$.
Under this isomorphism, the group $\stO(\Lambda_d(-1))$ is identified with $\stO(L_d^{\perp}) \cong H(L_d)$. Fix an embedding $L_d\hookrightarrow \Lambda$.
By Proposition \ref{MainResult}, we obtain the following commutative diagram:
\begin{equation}
\label{eq_diagram}    
\xymatrix@C=0em{\Omega(\Lambda_d(-1))\ar[d] \ar[rrrrr]^{\cong}  &&&&& \Omega(L_d^{\perp})\ar[d]\ar@{^{(}->}[rrrrr] &&&&& \Omega(\Lambda_{00})\ar[d]\\
 \Omega(\Lambda_d(-1))/\stO(\Lambda_d(-1))\ar[rrrr]^-{\cong} &&&& \Omega(L_d^{\perp})/H(L_d)= \!\!\!\!\! &\mD_{L_d}^{\lab}\ar[rrrr]^-{\nu} &&&& \mD_{L_d}\ar@{^{(}->}[r]&\mD
 }
\end{equation}
By Lemma \ref{lemma_mapnu}, the map $\nu$ is birational. It follows from the diagram above that there exists a birational map
\[\mD\supset\mD_{L_d}\dashrightarrow \Omega(\Lambda_d(-1))/\stO(\Lambda_d(-1))
\cong \Omega(\Lambda_d)/\stO(\Lambda_d).\]
In particular, we obtain a rational map
\begin{equation}
\label{eq_rationalmaponD}  
\mD_d\dashrightarrow\Omega(\Lambda_d)/\stO(\Lambda_d)
\end{equation}
which is birational when $d\equiv 4\mod 8$ and generically two-to-one when $d\equiv 2\mod 8$. Indeed, for a generic $x \in \mD_d'$, note that for some $g\in\tO(L_d^{\perp})$, $x$ and $(g\circ r)_{\mD}(x)$ are mapped to the same point in $\Omega(\Lambda_d)/\stO(\Lambda_d)$. 
The quotient $\Omega(\Lambda_d)/\stO(\Lambda_d)$ can be viewed as the moduli space of degree-$d$ quasi-polarized K3 surfaces \cite[Section~5]{HulekPloog}.
The map above induces a rational map
\begin{equation*}  
\mM_4\times_{\mD}\mD_d\dashrightarrow \Omega(\Lambda_d)/\stO(\Lambda_d)
\end{equation*}
sending a GM fourfold to an associated quasi-polarized K3 surface.

Now denote by $\tM_d$ the moduli space of polarized K3 surfaces of degree $d$. 
The period map induces an open immersion 
$\tM_d\hookrightarrow \Omega(\Lambda_d)/\stO(\Lambda_d)$.
When restricted to points outside $\mM_4\times_{\mD}\mD_8$, the image of the above rational map lies in $\tM_d$. We obtain a dominant rational map 
\begin{equation}
\label{eq_ratmap}
\gamma_d\colon\mM_4\times_{\mD}\mD_d\dashrightarrow \tM_d    
\end{equation}
sending a GM fourfold to an associated polarized K3 surface. 
This proves Theorem \ref{cor_ratmap}.
Note that $\gamma_d(X)$ is defined whenever $\rk\HH^{2,2}(X,\Z)=3$.

\begin{remark}
Note that $\gamma_d$ is not unique,
since the map $\mD_{L_d}\dashrightarrow \Omega(\Lambda_d)/\stO(\Lambda_d)$
depends on the choice of an isomorphism $L_d^{\perp}\cong\Lambda_d(-1)$.
To be precise, $\gamma_d$ is unique up to
$\tO(\Lambda_d)/\stO(\Lambda_d)$ when $d\equiv 4\mod 8$,
and up to $(\tO(\Lambda_d)/\stO(\Lambda_d))^2$ when $d\equiv 2\mod 8$.
\end{remark}

\subsection{Fibers of Fourier--Mukai partners}

We now apply the results in the previous sections to study Fourier--Mukai partners of GM fourfolds. In analogy to \cite{Huy, Pert1} for cubic fourfolds, we say that a \emph{Fourier--Mukai partner} of a GM fourfold $X$ is a GM fourfold $X'$ such that there exists an exact equivalence $\Ku(X) \xrightarrow{\sim} \Ku(X')$ of Fourier--Mukai type, i.e.\ the composition $\D(X) \to \Ku(X) \xrightarrow{\sim} \Ku(X') \to \D(X')$ has a Fourier--Mukai kernel. Note that by \cite[Theorem 1.6]{kuznetsov_perry_cones} non-isomorphic GM fourfolds in the same fiber of the period map are Fourier--Mukai partners. 

\medskip
Before proving Theorem \ref{thm_FMp}, we need to make the following remark. In analogy to \cite{AddTho}, the \emph{Mukai lattice} for $\Ku(X)$ has been defined in \cite[Section 3.1]{Pert2} as the abelian subgroup
$$\widetilde{\mathrm{H}}(\Ku(X),\Z):=\lbrace \kappa \in \mathrm{K}(X)_{\text{top}}: \chi([\mathcal{O}_X(i)],\kappa)=\chi([\mathcal{U}_X^*(i)],\kappa)=0 \, \text{ for }i=0,1 \rbrace$$
of the topological K-theory of $X$, with the Euler form $\chi$ with reversed sign and the weight-$2$ Hodge structure induced by pulling back via the isomorphism
$$\widetilde{\mathrm{H}}(\Ku(X),\Z) \otimes \C \rightarrow \HH^\bullet(X,\C)$$
given by the Mukai vector $v(-)=\text{ch}(-).\sqrt{\text{td}(X)}$. As a lattice, $\widetilde{\mathrm{H}}(\Ku(X),\Z) \cong U^{\oplus 4} \oplus E_8(-1)^{\oplus 2}$ by \cite[Theorem 1.2]{debarre_kuznetsov_2019}. We set
$$\widetilde{\mathrm{H}}^{1,1}(\Ku(X),\Z):= \widetilde{\mathrm{H}}^{1,1}(\Ku(X))\cap \widetilde{\mathrm{H}}(\Ku(X),\Z).$$
By \cite[Lemma 2.27]{kuznetsov_perry} there are two classes in  $\widetilde{\mathrm{H}}^{1,1}(\Ku(X),\Z)$ spanning a lattice $A_1^{\oplus 2}$. By \cite[Proposition 3.1]{Pert2}, there is a Hodge isometry
$$A_1^{\oplus 2\perp} \cong \HH^4(X,\Z)_{00},$$
up to a sign and a Tate twist, where the orthogonal complement is taken in $\widetilde{\mathrm{H}}(\Ku(X),\Z)$. 
In particular, a very general Hodge-special GM fourfold $X$ with discriminant $d$ has $\text{rk}\widetilde{\mathrm{H}}^{1,1}(\Ku(X),\Z)=3$ and $\Disc \widetilde{\mathrm{H}}^{1,1}(\Ku(X),\Z)=d$. Moreover, we have the following property.

\begin{lemma}
\label{lemma_eqinduceisometry}
Every equivalence $\emph{Ku}(X) \xrightarrow{\sim} \emph{Ku}(X')$ of Fourier--Mukai type induces a Hodge isometry $\widetilde{\mathrm{H}}(\emph{Ku}(X),\Z) \cong \widetilde{\mathrm{H}}(\emph{Ku}(X'),\Z)$.
\end{lemma}
\begin{proof}
Apply a similar argument as in \cite[Proposition 3.3]{Huy}.
\end{proof}
By the above lemma, every Fourier--Mukai partner of a very general Hodge-special GM fourfold with discriminant $d$ is a very general Hodge-special GM fourfold of the same discriminant.

\medskip
We are now ready to prove the next proposition which implies Theorem \ref{thm_FMp}.
Denote by $\tau(d)$ the number of distinct primes that divide $d/2$.
\begin{proposition}
\label{prop_FMpGM}
Let $d$ be a positive integer satisfying condition \eqref{eq_astast}. If $X$ is a very general Hodge-special GM fourfold with discriminant $d \equiv 4 \mod 8$ (resp.\ $d \equiv 2 \mod 8$), then there are $2^{\tau(d)-1}$ (resp.\ $2^{\tau(d)}$) fibers of the period map $p$ such that, when non-empty, their elements are Fourier--Mukai partners of $X$. Moreover, all Fourier--Mukai partners of $X$ are obtained in this way.
\end{proposition}
\begin{proof}
We fix a choice of the rational map $\gamma_d\colon\mM_4\times_{\mD}\mD_d \dashrightarrow \tM_d $ of Section \ref{DefRationalMap}.
Let $X$ be a GM fourfold as in the statement and consider $\gamma_d(X)=(S,l)$, a degree-$d$ polarized K3 surface associated to $X$. Note that $S$ has Picard rank $1$. Moreover, by \cite[Theorem 3.6]{Pert2} and \cite[Theorem 1.9]{PPZ} there exists an exact equivalence $\Ku(X) \xrightarrow{\sim} \D(S)$. By \cite[Proposition 1.10]{Oguiso}, $S$ has $m:=2^{\tau(d)-1}$ non-isomorphic Fourier--Mukai partners. 

Choose $m$ K3 surfaces $S_1:=S, S_2 \dots, S_m$ as representatives for each isomorphism class of Fourier--Mukai partners endowed with the unique degree-$d$ polarizations $l_1:=l, \dots, l_m$, respectively. These polarized K3 surfaces determine $m$ distinct points in $\Omega(\Lambda_d)/\stO(\Lambda_d)$, which we still denote by $(S_i,l_i)$ for $1 \leq i \leq m$. 
As summarized in diagram \eqref{eq_diagram}, their image via \eqref{eq_rationalmaponD} defines $m$ (resp.\ $2m$) period points in $\mD_d$ if $d \equiv 4 \mod 8$ (resp.\ $d \equiv 2 \mod 8$). 
We denote by $x_i \in \mD_d$ the period point defined by $(S_i,l_i)$ if $d \equiv 4 \mod 8$, and by $x_i' \in \mD_d'$, $x_i'' \in \mD_d''$ those defined by $(S_i,l_i)$ if $d \equiv 2 \mod 8$. 
Assume that $x_i$ (resp.\ $x_i'$ or $x_i''$) is in the image of the period map $p$ and consider a GM fourfold $X'$ in the fiber of $p$ over this point. 
Then $$\Ku(X') \xrightarrow{\sim} \D(S_i) \xrightarrow{\sim} \D(S) \xrightarrow{\sim} \Ku(X).$$
Finally, by Lemma \ref{lemma_eqinduceisometry}, if $X'$ is a Fourier--Mukai partner of $X$, then $X'$ is a very general Hodge-special GM fourfold of the same discriminant. 
Thus $\gamma_d(X')$ is a well defined element in $\tM_d$. But then $\gamma_d(X')$ is a degree-$d$ polarized K3 surface which is a Fourier--Mukai partner of $(S,l)$, hence isomorphic to $(S_i,l_i)$ for a certain $1 \leq i \leq m$. This implies the statement.
\end{proof}

\begin{remark}
Note that the image of the period map is not known \cite[Question 9.1]{debarre_iliev_manivel_2015}. Thus some of the fibers of Proposition \ref{prop_FMpGM} could a priori be empty.
\end{remark}

\section{Gushel--Mukai fourfolds and twisted K3 surfaces}
\label{section_GMandtwistedK3}
 In this section we recall the definition of the moduli spaces of twisted polarized K3 surfaces with fixed order and degree introduced in \cite{BrakkeeTwistedK3}, and then we prove Theorem \ref{thm_ratmaptwisted} and Theorem \ref{thm_FMptwisted}.

\subsection{Moduli and periods of twisted K3 surfaces}
We summarize the relevant results of \cite{BrakkeeTwistedK3}.
Recall that for a complex K3 surface $S$, the Brauer group $\br(S)$ is isomorphic to the cohomological Brauer group 
\[\HH^2_{\text{\'et}}(S,\G_m)\cong\HH^2(S,\mO_S^*)_{\tors}\cong (\Q/\Z)^{\oplus 22-\rho(S)}.\]
Let $T(S):=\ns(S)^{\perp}\subset\HH^2(S,\Z)$ be the transcendental lattice of $S$. Then there is an isomorphism
\[\br(S)\cong\hom(T(S),\Q/\Z).\]
We denote by $\br(S)[r]$ the group of elements in $\br(S)$ whose order divides $r$. There exists a surjection
\[\widetilde{\br}(S)[r]:=\hom(\HH^2(S,\Z)_{\prim},\Z/r\Z)\twoheadrightarrow \hom(T(S),\Z/r\Z)\cong \br(S)[r]\]
which is an isomorphism if and only if $\rho(S)=1$.

\begin{theorem}[{\cite[Theorem~1]{BrakkeeTwistedK3}}]
 There exists a scheme $\tM_d[r]$ which is a coarse moduli space for triples $(S,l,\alpha)$ consisting of a polarized K3 surface $(S,l)$ of degree $d$ and an element $\alpha\in\widetilde{\br}(S)[r]$.
 There exists a subscheme $\tM_d^r\subset\tM_d[r]$ which is a coarse moduli space for those triples for which $\alpha$ has order $r$.
\end{theorem}

The spaces $\tM_d[r]$ and $\tM_d^r$ are constructed as follows.
Let $\tM_d^{\mar}$ be the (fine) moduli space of triples $(S,l,\varphi)$
where $(S,l)$ is as before and $\varphi$ is an isomorphism
$\HH^2(S,\Z)_{\prim}\cong\Lambda_d$. Note that $\varphi$ induces an isomorphism $\varphi_r\colon \br(S)[r]\to\hom(\Lambda_d,\Z/r\Z)$. The group $\stO(\Lambda_d)$ induces an action on $\hom(\Lambda_d,\Z/r\Z)$,
and on the product
\[\tM_d^{\mar}[r]:=\tM_d^{\mar}\times\hom(\Lambda_d,\Z/r\Z)\]
by
\[g(S,l,\varphi,\alpha) = (S,l,g\circ\varphi, \varphi^{-1}_rg\varphi_r(\alpha)).\]
The space $\tM_d[r]$ is the quotient $\tM_d^{\mar}[r]/\stO(\Lambda_d)$.

For $w\in\hom(\Lambda_d,\Z/r\Z)$,
denote by $\stab(w)\subset\stO(\Lambda_d)$ its stabilizer under the action of $\stO(\Lambda_d)$. 
Then $\tM_d[r]$ is a disjoint union
$\coprod_{[w]}\tM_w$
where $[w]\in \hom(\Lambda_d,\Z/r\Z)/\stO(\Lambda_d)$
and
\[\tM_w=(\tM_d^{\mar}\times\{w\})/\stab(w).\]
Each component $\tM_w$ is an irreducible quasi-projective variety with at most finite quotient singularities \cite[Corollary~2.2]{BrakkeeTwistedK3}.
It parametrizes triples $(S,l,\alpha)$ that admit a marking $\varphi$ such that $\varphi_r(\alpha) = w$.
The space $\tM_d^r$ is the union of those $\tM_w$ for which $w$ has order $r$.

\medskip
Given $(S,l)\in\tM_d$ and $\alpha\in\widetilde{\br}(S)[r] \cong\frac1r\HH^2(S,\Z)^{\dual}_{\prim}/\HH^2(S,\Z)_{\prim}^{\dual}$, there is an associated Hodge structure $\widetilde{\HH}(S,\alpha,\Z)$ of K3 type on the full cohomology $\HH^*(S,\Z)$ of $S$. 
Namely, fix a lift of $\alpha$ to $\frac1r\HH^2(S,\Z)_{\prim}^{\dual}\subset\HH^2(S,\Q)$, that we will also denote by $\alpha$.
Then $\widetilde{\HH}(S,\alpha,\Z)$ is defined by
\[\widetilde{\HH}{}^{2,0}(S,\alpha):=\C[\sigma+\alpha\wedge\sigma]\subset\HH^*(S,\Z),\]
where $\sigma$ is a non-degenerate holomorphic 2-form on $S$.
If $\alpha$ maps to $\alpha'$ under $\widetilde{\br}(S)[r]\twoheadrightarrow\br(S)[r]$, then $\widetilde{\HH}(S,\alpha,\Z)$ is isomorphic to the Hodge structure $\widetilde{\HH}(S,\alpha',\Z)$ defined by $\alpha'$ as in \cite[Section~4]{GeneralizedCY}.

\medskip
Denote by $\widetilde{\Lambda}$ the extended K3 lattice. 
Let $w\in\hom(\Lambda_d,\Z/r\Z)$ and denote by $T_w$ the finite-index sublattice $\ker(w)\subset\Lambda_d$.
Using the above one shows that, up to an identification of $T_w$ with $\exp(w)\Lambda_d\cap\widetilde{\Lambda}$ (see \cite[Section~3.1]{BrakkeeTwistedK3}), there is a holomorphic, injective period map
\begin{align*} 
 \tM_d^{\mar}\times\{w\}&\to\Omega(T_w)\\
 (S,l,\varphi,w)&\mapsto \widetilde{\varphi}\left(\widetilde{\HH}{}^{2,0}(S,\varphi_r^{-1}(w)\right).
\end{align*}
It induces an algebraic embedding
$\tM_w\hookrightarrow\Omega(T_w)/\stab(w)$.

\medskip
For later use, we define the \emph{Picard group} of a twisted K3 surface as
\[\pic(S,\alpha):=\widetilde{\HH}{}^{1,1}(S,\alpha)\cap \widetilde{\HH}(S,\alpha,\Z)\] and its \emph{transcendental lattice} $T(S,\alpha)$ as the orthogonal complement of $\pic(S,\alpha)$ in $\widetilde{\HH}(S,\alpha,\Z)$. When $\alpha$ is trivial, we have $\pic(S,\alpha) = \HH^0(S,\Z)\oplus\HH^{1,1}(S,\Z)\oplus\HH^4(S,\Z)$ and $T(S,\alpha) = T(S)$. One can show that there is an isomorphism of lattices $T(S,\alpha)\cong \ker(\alpha\colon T(S)\to\Q/\Z)$.

\subsection{Twisted K3 surfaces associated to GM fourfolds}
\label{section_GMvstwistedK3}
Recall \cite[Definition~3.11]{Pert2} that if $X$ is a Hodge-special GM fourfold, a twisted K3 surface $(S,\alpha)$ is said to be associated to $X$ when
there is a Hodge isometry 
\[\widetilde{\HH}(\Ku(X),\Z)\cong \widetilde{\HH}(S,\alpha,\Z).\]
Note that if $d$ is the degree and $r$ the order of $(S,\alpha)$,
then $X$ has discriminant $dr^2$.
One can show \cite[Theorem 1.1]{Pert2} that $X$ has an associated twisted K3 surface if and only if the period point of $X$ lies in $\mD_{d'}$ for some $d'$ satisfying
\begin{equation}\tag{$\ast\ast'$}
d' = \prod_i p_i^{n_i} \text{ with } n_i\equiv 0 \mod 2\text{ for }p_i\equiv 3\mod 4
\end{equation}
where the $p_i$ are distinct primes. Note that this is equivalent to the following: $d'$ is of the form $dr^2$ for some integers $d$ and $r$, where $d$ satisfies $(\ast\ast)$. This decomposition $d'=dr^2$ is however not unique.

We will prove that $(\ast\ast')$ is equivalent to a condition on the lattice $L_{d'}^{\perp}$.
We need the following lemma. For $w\in\hom(\Lambda_d,\Z/r\Z)$ of order $r$, let $T_w\subset\Lambda_d$ be the index-$r$ sublattice $\ker(w)$.
We embed $T_w$ primitively into $\widetilde{\Lambda}$ using the map $\exp(w)$ (see \cite[Section~3.1]{BrakkeeTwistedK3}).
\begin{lemma}\label{OrthGpSurj}
 Let $S_w:=T_w^{\perp}\subset\widetilde{\Lambda}$. The canonical map 
 $\tO(S_w)\to\tO(\Disc S_w)$ is surjective.
\end{lemma}
\begin{proof}
Note that $S_w$ has rank 3 and its discriminant group is isomorphic to $\Disc T_w$.
By \cite[Proposition~6.5]{debarre_iliev_manivel_2015}, this group is either cyclic or isomorphic to $(\Z/2\Z)^2\times\Z/(d'/4)\Z$. In the first case, the statement follows from \cite[Theorem~1.14.2]{Nikulin}. In the second case, it follows from \cite[Corollary~VIII.7.8]{MiMo}.
\end{proof}

\begin{cor}\label{EquivAssociated}
Consider an integer $d'>8$ with $d' \equiv 0,2,4 \mod 8$. Then $d'$ satisfies $(\ast\ast')$ if and only if for some decomposition $d'=dr^2$ with $d$ satisfying $(\ast\ast)$, there is a $w\in\hom(\Lambda_d,\Z/r\Z)$ and a lattice isometry $L_{d'}^{\perp}(-1)\cong T_w$.
\end{cor}
\begin{proof}
 Suppose $d'$ satisfies $(\ast\ast')$. Let $X$ be a very general GM fourfold with period point in $\mD_{d'}$, which exists by \cite[Theorem 8.1]{debarre_iliev_manivel_2015},
 and let $(S,\alpha,\Z)$ be a twisted K3 surface associated to $X$.
 Let $d$ be the degree of $S$ and $r$ the order of $\alpha$, so $d'=dr^2$.
 Then the Hodge isometry 
 $\widetilde{\HH}(\Ku(X),\Z)\cong\widetilde{\HH}(S,\alpha,\Z)$ induces a lattice isometry of the transcendental parts:
 \[L_{d'}^{\perp}(-1)\cong T(S,\alpha)\cong\ker(\alpha\colon\HH^2(S,\Z)_{\prim}\to\Z/r\Z).\]
Now any marking $\HH^2(S,\Z)_{\prim}\cong\Lambda_d$ induces an isometry
 $L_{d'}^{\perp}(-1)\cong T_w$ for some $w\in\hom(\Lambda_d,\Z/r\Z)$.
 
 Vice versa, assume we have an isometry $L_{d'}^{\perp}(-1)\cong T_w$ as above. Then the associated period domains $\Omega(L_{d'}^{\perp}(-1))\cong \Omega(L_{d'}^{\perp})$ and $\Omega(T_w)$ are also isomorphic. It follows that if $X$ is a very general GM fourfold of discriminant $d'$, then $L_{d'}^{\perp}\subset\HH^4(X,\Z)$ is Hodge isometric, up to a sign and a Tate twist, to $T(S,\alpha)$ for some twisted K3 surface $(S,\alpha)$ of degree $d$ and order $r$. By Lemma \ref{OrthGpSurj}, this can be extended to a Hodge isometry
 $\widetilde{\HH}(\Ku(X),\Z)\cong \widetilde{\HH}(S,\alpha,\Z)$. Since $X$ is very general in $\mD_{d'}$, so its period lies in $\mD_e$ if and only if $e=d'$, it follows that $d'$ satisfies $(\ast\ast')$.
\end{proof}

Assume we are in the situation of Corollary \ref{EquivAssociated}, so $d$ satisfies $(\ast\ast')$ and we have a fixed $w\in\hom(\Lambda_d,\Z/r\Z)$ such that $L_{d'}^{\perp}(-1)$ is isomorphic to $T_w$.
This induces an isomorphism
$\Omega(T_w)/\stO(T_w)\cong\Omega(L_{d'}^{\perp})/H(L_{d'})=\mD_{L_{d'}}^{\lab}$.
Next, note that $\stO(T_w)$ is a subgroup of $\stab(w)$ \cite[Lemma~4.1]{BrakkeeTwistedK3}.
Summarized in a diagram, we have

\begin{equation}   
\label{cd_twistedperiod}
\xymatrix@C=0em{&&& \Omega(T_w)\ar[d] \ar[rrr]^{\cong}  &&& \Omega(L_{d'}^{\perp})\ar[d]\ar@{^{(}->}[rrrrr] &&&&& \Omega(\Lambda_{00})\ar[d]\\
 &&& \Omega(T_w)/\stO(T_w)\ar[d]_{\pi}\ar[rrr]^-{\cong} &&& \mD_{L_{d'}}^{\lab}\ar[rrrr]^-{\nu} &&&& \mD_{L_{d'}}\ar@{^{(}->}[r]&\mD \\
 \tM_w\ar@{^{(}->}[rrr] &&& \Omega(T_w)/\stab(w) &&&&&&&&&
 }
\end{equation}
where $\pi\colon\Omega(T_w)/\stO(T_w)\to \Omega(T_w)/\stab(w)$ is a finite map. As in Section \ref{DefRationalMap}, we obtain a finite dominant rational map $\mD_{L_{d'}}\dashrightarrow\tM_w$.
\begin{cor}\label{RatMapTwisted}
There is a dominant rational map 
 \[\delta_{d'}\colon \mM_4\times_{\mD}\mD_{d'}\dashrightarrow\tM_w\]
which sends a very general Hodge-special GM fourfold $X$ of discriminant $d'$ to a polarized twisted K3 surface associated to $X$.
\end{cor}
The map is defined whenever $\rk\HH^{2,2}(X,\Z)=3$.
When $\rk \HH^{2,2}(X,\Z)>3$
and the map is defined at $X$, then the image of $X$ is a triple $(S,l,\alpha)\in \tM_w$ such that if $\alpha'\in\br(S)[r]$ is the image of $\alpha\in\widetilde{\br}(S)[r]$, then $(S,l,\alpha')$ is associated to $X$. 
Namely, by Lemma \ref{OrthGpSurj}, the Hodge isometry
\[\HH^4(X,\Z)_{00}\supset L_{d'}^{\perp}\cong T_w\subset \widetilde{\HH}(S,\alpha,\Z)\] 
extends to a Hodge isometry
$\widetilde{\HH}(\Ku(X),\Z)\cong\widetilde{\HH}(S,\alpha,\Z)=\widetilde{\HH}(S,\alpha',\Z)$.

\subsection{Fourier--Mukai partners in the twisted case}
In this section we apply Corollary \ref{RatMapTwisted} to construct Fourier--Mukai partners of a very general GM fourfold with a twisted associated K3 surface. 

First, we need the following lemma, which is the analogue of \cite[Lemma 2.3]{Huy} in the case of cubic fourfolds.

\begin{lemma}
\label{lemma_inverseorientation}
The Mukai lattice $\widetilde{\HH}(\emph{Ku}(X),\Z)$ of a GM fourfold $X$ has an orientation reversing Hodge isometry.
\end{lemma}
\begin{proof}
Denote by $\lambda_1$ and $\lambda_2$ the standard generators of $A_1^{\oplus 2} \subset \widetilde{\HH}(\Ku(X),\Z)$. Consider the isometry $g \in \tO(A_1^{\oplus 2})$ defined by
$$g(\lambda_1)=-\lambda_1 \text{ and }g(\lambda_2)=\lambda_2.$$
Since $g$ acts trivially on the discriminant group of $A_1^{\oplus 2}$, by \cite[Prop. 1.6.1 and Cor. 1.5.2]{Nikulin} there is an isometry $\tilde{g}$ of $\widetilde{\Lambda}:=U^{\oplus 4} \oplus E_8(-1)^{\oplus 2}$ extending $g$ and acting trivially on $A_1^{\oplus 2 ^\perp}$. By definition $\tilde{g}$ reverses the orientation of the two positive directions in $A_1^{\oplus 2}$ and preserves the orientation of the two positive directions in $A_1^{\oplus 2 \perp}$. Moreover, $\tilde{g}$ preserves the Hodge structure as it acts trivially on $A_1^{\oplus 2 \perp}$. This implies the statement.
\end{proof}

\begin{remark}
Note that there is an autoequivalence of $\Ku(X)$ which induces the Hodge isometry described in Lemma \ref{lemma_inverseorientation}. Indeed, consider the composition $\mathbb{L}_{\langle \mathcal{O}_X, \mathcal{U}_X^*,\mathcal{O}_X(1) \rangle} \circ  (\mathbb{D}(-) \otimes \mathcal{O}_X(1))$, where $\mathbb{D}(-):=\textrm{R}Hom(-,\mathcal{O}_X)$ and $\mathbb{L}_{\langle \mathcal{O}_X, \mathcal{U}_X^*,\mathcal{O}_X(1) \rangle}$ is the left mutation functor through $\mathcal{O}_X, \mathcal{U}_X^*,\mathcal{O}_X(1)$. One can check that this composition induces an autoequivalence when restricted to $\Ku(X)$, acting on the Mukai lattice as required.
\end{remark}

We can now prove the following proposition which implies Theorem \ref{thm_FMptwisted}. Denote by $\varphi(r)$ the Euler function evaluated in $r$. Recall that a Fourier--Mukai partner of order $r$ of a twisted K3 surface $(S,\alpha)$ is a twisted K3 surface $(S',\alpha')$ with $\alpha'$ of order $r$ such that there is an equivalence $\D(S,\alpha) \xrightarrow{\sim} \D(S',\alpha')$.

\begin{proposition}
\label{prop_FMptwistedcase}
Let $d'=dr^2$ be a positive integer such that a very general GM fourfold of discriminant $d'$ admits an associated polarized twisted K3 surface of degree $d$ and order $r$. If $d' \equiv 0 \mod 4$ (resp.\ $d' \equiv 2 \mod 8$), then there are $m'$ (resp.\ $2m'$) fibers of the period map $p$ such that, when non-empty, their elements are Fourier--Mukai partners of $X$, where
\begin{equation}
\label{eq_m'}
m'=
\begin{cases}
\varphi(r)2^{\tau(d)-1} & \text{if } r=2 \text{ or } d>2 \\
\varphi(r)/ 2  & \text{if } r>2 \text{ and } d=2.
\end{cases}
\end{equation}
\end{proposition}
\begin{proof}
We fix a rational map $\delta_{d'}\colon \mM_4\times_{\mD}\mD_{d'}\dashrightarrow\tM_w$ as in Corollary \ref{RatMapTwisted}. 
Let $X$ be a GM fourfold as in the statement and consider the twisted degree-$d$ polarized K3 surface $\delta_{d'}(X)=(S,\alpha)$ with $\text{ord}(\alpha)=r$. Note that $S$ has Picard rank $1$ and by \cite[Theorem 1.1]{Pert2} and \cite[Theorem 1.9]{PPZ} there is an equivalence $\Ku(X) \xrightarrow{\sim} \D(S,\alpha)$. Let $m'$ be the number of Fourier--Mukai partners of $(S,\alpha)$ of order $r$. 
By Lemmas \ref{OrthGpSurj} and \ref{lemma_inverseorientation}, arguing as in \cite[Proposition 4.4]{Pert1}, one can show that $m'$ is equal to the upper bound given in \cite[Proposition 4.3]{Ma}. Moreover, by \cite[Proposition 4.7]{Pert1} this number is given by \eqref{eq_m'} as in the statement. Then using diagram \eqref{cd_twistedperiod} and arguing as in Proposition \ref{prop_FMpGM}, we deduce the statement.
\end{proof}

\begin{remark}
Note that a GM fourfold as in Proposition \ref{prop_FMptwistedcase} could have other Fourier--Mukai partners. Indeed, they could be obtained from Fourier--Mukai partners of $(S,\alpha)$ with order different from $r$.
\end{remark}

\begin{remark}
The construction of the rational map in \cite[Section~4]{BrakkeeTwistedK3} can be used in the case of cubic fourfolds to simplify the proof of \cite[Theorem 1.2]{Pert1}. More precisely, the rational map allows to skip the computation in \cite[Section 4.1]{Pert2}. As a consequence, it is possible to remove the assumption in \cite[Theorem 1.2]{Pert1} that $9$ does not divide the discriminant, giving a more complete statement.
\end{remark}

\bibliography{MarkedLabelledGM4}
\bibliographystyle{alpha}

Korteweg--de Vries Institute, University of Amsterdam, P.O. Box 94248, 1090 GE Amsterdam, Netherlands.\\
\indent E-mail address: \texttt{e.l.brakkee@uva.nl}\\
\indent URL: \texttt{https://staff.fnwi.uva.nl/e.l.brakkee/}\\ 

Dipartimento di Matematica ``F.\ Enriques'', Universit\`a degli Studi di Milano, Via Cesare Saldini 50, 20133 Milano, Italy. \\
\indent E-mail address: \texttt{laura.pertusi@unimi.it}\\
\indent URL: \texttt{http://www.mat.unimi.it/users/pertusi} 

\end{document}